\documentclass[12pt]{amsart}
\usepackage{amssymb,latexsym}    
\usepackage{fullpage}

\input{xypic}
\newcommand{\Mdef}[2]{\newcommand{#1}{\relax \ifmmode #2 \else $#2$\fi}}

\newcommand{\im}{\mathrm{im}}

\newcommand{\sm }{\wedge}

\newcommand{\Hom}{\mathrm{Hom}}

\newcommand{\Ext}{\mathrm{Ext}}
\Mdef{\bhom}{\mathbf{\hat{H}om}}

\Mdef{\Mod}{\mathrm{mod}}

\newcommand{\st}{\; | \;}

\newtheorem{thm}{Theorem}[section]
\newtheorem{lemma}[thm]{Lemma}
\newtheorem{prop}[thm]{Proposition}
\newtheorem{cor}[thm]{Corollary}

\theoremstyle{definition}

\newcommand{\qqed}{\qed \\[1ex]}
\renewenvironment{proof}[1][\hspace*{-.8ex}]{\noindent {\bf Proof #1:\;}}{\qqed}

\Mdef{\PH} {\Phi^H}
\Mdef{\PK} {\Phi^K}
\Mdef{\PL} {\Phi^L}
\Mdef{\PT} {\Phi^{\T}}

\Mdef{\ef}{E{\cF}_+}
\Mdef{\etf}{\widetilde{E}{\cF}}
\Mdef{\eg}{E{G}_+}
\Mdef{\etg}{\tilde{E}{G}}

\Mdef{\infl}{\mathrm{inf}}
\Mdef{\defl}{\mathrm{def}}
\Mdef{\res}{\mathrm{res}}
\Mdef{\ind}{\mathrm{ind}}
\Mdef{\coind}{\mathrm{coind}}

\Mdef{\univ}{\mathcal{U}}

\Mdef{\Fp}{\mathbb{F}_p}
\Mdef{\Zpinfty}{\Z /p^{\infty}}
\Mdef{\Zpadic}{\Z_p^{\wedge}}

\newcommand{\bi}{\begin{itemize}}
\newcommand{\be}{\begin{enumerate}}
\newcommand{\bc}{\begin{center}}
\newcommand{\bd}{\begin{description}}
\newcommand{\ei}{\end{itemize}}
\newcommand{\ee}{\end{enumerate}}
\newcommand{\ec}{\end{center}}
\newcommand{\ed}{\end{description}}

\newcommand{\adjunction}[4]{
\diagram
#1:#2 \rrto<0.7ex> &&
#3  \llto<0.7ex> :#4 
\enddiagram}
\newcommand{\lra}{\longrightarrow}
\newcommand{\lla}{\longleftarrow}

\newcommand{\iso}{\cong}

\Mdef{\we}{\mathbf{we}}
\Mdef{\fib}{\mathbf{fib}}
\Mdef{\cof}{\mathbf{cof}}
\Mdef{\BI}{\mathcal{BI}}

\Mdef{\A}{\mathbb{A}}
\Mdef{\B}{\mathbb{B}}
\Mdef{\C}{\mathbb{C}}
\Mdef{\D}{\mathbb{D}}
\Mdef{\E}{\mathbb{E}}
\Mdef{\T}{\mathbb{T}}
\Mdef{\F}{\mathbb{F}}
\Mdef{\G}{\mathbb{G}}
\Mdef{\I}{\mathbb{I}}
\Mdef{\N}{\mathbb{N}}
\Mdef{\Q}{\mathbb{Q}}
\Mdef{\R}{\mathbb{R}}
\Mdef{\bbS}{\mathbb{S}}
\Mdef{\Z}{\mathbb{Z}}

\Mdef{\bA}{\mathbb{A}}
\Mdef{\bB}{\mathbb{B}}
\Mdef{\bC}{\mathbb{C}}
\Mdef{\bD}{\mathbb{D}}
\Mdef{\bE}{\mathbb{E}}
\Mdef{\bF}{\mathbb{F}}
\Mdef{\bG}{\mathbb{G}}
\Mdef{\bH}{\mathbb{H}}
\Mdef{\bI}{\mathbb{I}}
\Mdef{\bJ}{\mathbb{J}}
\Mdef{\bK}{\mathbb{K}}
\Mdef{\bL}{\mathbb{L}}
\Mdef{\bM}{\mathbb{M}}
\Mdef{\bN}{\mathbb{N}}
\Mdef{\bO}{\mathbb{O}}
\Mdef{\bP}{\mathbb{P}}
\Mdef{\bQ}{\mathbb{Q}}
\Mdef{\bR}{\mathbb{R}}
\Mdef{\bS}{\mathbb{S}}
\Mdef{\bT}{\mathbb{T}}
\Mdef{\bU}{\mathbb{U}}
\Mdef{\bV}{\mathbb{V}}
\Mdef{\bW}{\mathbb{W}}
\Mdef{\bX}{\mathbb{X}}
\Mdef{\bY}{\mathbb{Y}}
\Mdef{\bZ}{\mathbb{Z}}

\Mdef{\cA}{\mathcal{A}}
\Mdef{\cB}{\mathcal{B}}
\Mdef{\cC}{\mathcal{C}}
\Mdef{\mcD}{\mathcal{D}} 
\Mdef{\cE}{\mathcal{E}}
\Mdef{\cF}{\mathcal{F}}
\Mdef{\cG}{\mathcal{G}}
\Mdef{\mcH}{\mathcal{H}} 
\Mdef{\cI}{\mathcal{I}}
\Mdef{\cJ}{\mathcal{J}}
\Mdef{\cK}{\mathcal{K}}
\Mdef{\mcL}{\mathcal{L}}

\Mdef{\cM}{\mathcal{M}}
\Mdef{\cN}{\mathcal{N}}
\Mdef{\cO}{\mathcal{O}}
\Mdef{\cP}{\mathcal{P}}
\Mdef{\cQ}{\mathcal{Q}}
\Mdef{\mcR}{\mathcal{R}}
\Mdef{\cS}{\mathcal{S}}
\Mdef{\cT}{\mathcal{T}}
\Mdef{\cU}{\mathcal{U}}
\Mdef{\cV}{\mathcal{V}}
\Mdef{\cW}{\mathcal{W}}
\Mdef{\cX}{\mathcal{X}}
\Mdef{\cY}{\mathcal{Y}}
\Mdef{\cZ}{\mathcal{Z}}

\Mdef{\At}{\tilde{A}}
\Mdef{\Bt}{\tilde{B}}
\Mdef{\Ct}{\tilde{C}}
\Mdef{\Et}{\tilde{E}}
\Mdef{\Ht}{\tilde{H}}
\Mdef{\Kt}{\tilde{K}}
\Mdef{\Lt}{\tilde{L}}
\Mdef{\Mt}{\tilde{M}}
\Mdef{\Nt}{\tilde{N}}
\Mdef{\Pt}{\tilde{P}}

\Mdef{\tA}{\tilde{A}}
\Mdef{\tB}{\tilde{B}}
\Mdef{\tC}{\tilde{C}}
\Mdef{\tE}{\tilde{E}}
\Mdef{\tH}{\tilde{H}}
\Mdef{\tK}{\tilde{K}}
\Mdef{\tL}{\tilde{L}}
\Mdef{\tM}{\tilde{M}}
\Mdef{\tN}{\tilde{N}}
\Mdef{\tP}{\tilde{P}}

\Mdef{\ft}{\tilde{f}}
\Mdef{\xt}{\tilde{x}}
\Mdef{\yt}{\tilde{y}}

\Mdef{\Ab}{\overline{A}}
\Mdef{\Bb}{\overline{B}}
\Mdef{\Cb}{\overline{C}}
\Mdef{\Db}{\overline{D}}
\Mdef{\Eb}{\overline{E}}
\Mdef{\Fb}{\overline{F}}
\Mdef{\Gb}{\overline{G}}
\Mdef{\Hb}{\overline{H}}
\Mdef{\Ib}{\overline{I}}
\Mdef{\Jb}{\overline{J}}
\Mdef{\Kb}{\overline{K}}
\Mdef{\Lb}{\overline{L}}
\Mdef{\Mb}{\overline{M}}
\Mdef{\Nb}{\overline{N}}
\Mdef{\Ob}{\overline{O}}
\Mdef{\Pb}{\overline{P}}
\Mdef{\Qb}{\overline{Q}}
\Mdef{\Rb}{\overline{R}}
\Mdef{\Sb}{\overline{S}}
\Mdef{\Tb}{\overline{T}}
\Mdef{\Ub}{\overline{U}}
\Mdef{\Vb}{\overline{V}}
\Mdef{\Wb}{\overline{W}}
\Mdef{\Xb}{\overline{X}}
\Mdef{\Yb}{\overline{Y}}
\Mdef{\Zb}{\overline{Z}}

\Mdef{\db}{\overline{d}}
\Mdef{\hb}{\overline{h}}
\Mdef{\qb}{\overline{q}}
\Mdef{\rb}{\overline{r}}
\Mdef{\tb}{\overline{t}}
\Mdef{\ub}{\overline{u}}
\Mdef{\vb}{\overline{v}}

\Mdef{\hc}{\hat{c}}
\Mdef{\he}{\hat{e}}
\Mdef{\hf}{\hat{f}}
\Mdef{\hA}{\hat{A}}
\Mdef{\hH}{\hat{H}}
\Mdef{\hJ}{\hat{J}}
\Mdef{\hM}{\hat{M}}
\Mdef{\hP}{\hat{P}}
\Mdef{\hQ}{\hat{Q}}

\Mdef{\thetab}{\overline{\theta}}
\Mdef{\phib}{\overline{\phi}}

\Mdef{\uA}{\underline{A}}
\Mdef{\uB}{\underline{B}}
\Mdef{\uC}{\underline{C}}
\Mdef{\uD}{\underline{D}}

\Mdef{\bolda}{\mathbf{a}}
\Mdef{\boldb}{\mathbf{b}}
\Mdef{\boldD}{\mathbf{D}}

\Mdef{\fm}{\frak{m}}
\Mdef{\fp}{\frak{p}}

\Mdef{\eps}{\epsilon}

\renewcommand{\Et}{\cE_t}

\newcommand{\M}{\bM}

\newcommand{\modcat}[1]{\mbox{$#1$-mod}}
\newcommand{\torsmodcat}[1]{\mbox{tors-$#1$-mod}}

\newcommand{\cellmodcat}[1]{\mbox{cell-$#1$-mod}}

\newcommand{\modcatG}[1]{\mbox{$#1$-mod-$G$-spectra}}
\newcommand{\cellmodcatG}[1]{\mbox{cell-$#1$-mod-$G$-spectra}}
\newcommand{\modcatW}[1]{\mbox{$#1$-mod-$W$-spectra}}
\newcommand{\cellmodcatW}[1]{\mbox{cell-$#1$-mod-$W$-spectra}}

\newcommand{\cellmodcatQW}[1]{\mbox{cell-$#1[W]$-mod}}

\newcommand{\freeGspectra}{\mbox{free-$G$-spectra}}

\newcommand{\HBN}{H^*(\widetilde{B}N)}
\newcommand{\HBNW}{H^*(\widetilde{B}N)[W]}
\newcommand{\CBN}{C^*(\widetilde{B}N)}
\newcommand{\CBNW}{C^*(\widetilde{B}N)[W]}
\newcommand{\DBNp}{D\widetilde{B}N_+}
\newcommand{\symm}{\mathrm{Symm}}

\begin{document}
\title{ An algebraic model for free rational $G$-spectra.} 
\author{J.~P.~C.~Greenlees}
\address{Department of Pure Mathematics, The Hicks Building, 
Sheffield S3 7RH. UK.}
\email{j.greenlees@sheffield.ac.uk}
\author{B.~Shipley}
\thanks{The first author is grateful for support under EPSRC grant
  number EP/H040692/1. 
  This material is based upon work by the second author supported by the National Science Foundation under Grant No. DMS-1104396.}
\address{Department of Mathematics, Statistics and Computer Science, University of Illinois at
Chicago, 510 SEO m/c 249,
851 S. Morgan Street,
Chicago, IL, 60607-7045, USA}
\email{shipleyb@uic.edu}
\date{}


\begin{abstract}
We show that for any compact Lie group $G$ with identity component $N$
and component group $W=G/N$, the category of free
rational $G$-spectra is equivalent to the category of torsion modules
over the twisted group ring $H^*(BN)[W]$. This gives an algebraic
classification of rational $G$-equivariant cohomology theories on free
$G$-spaces and a practical method for calculating the groups of
natural transformations between them. 
\end{abstract}

\maketitle
\tableofcontents

\section{Introduction}
\subsection{Context}
In algebraic topology, one of the basic invariants of $G$-spaces is an
equivariant cohomology theory, $E_G^*(\cdot )$ (one thinks first of
various types of equivariant $K$-theory and equivariant cobordism, but
also of Bredon and Borel cohomology). Such a cohomology theory  is in particular a contravariant functor to an abelian
category satisfying the Eilenberg-Steenrod and Milnor axioms (which is to say it
is homotopy invariant, is an exact functor, has a Mayer-Vietoris sequence,
and takes sums to products). In general we impose restrictions on
behaviour under suspensions, but  to simplify the situation somewhat, we
restrict the cohomology theory to free $G$-spaces, which makes the
other restrictions unnecessary. A  formal stabilisation process constructs a category (``free
$G$-spectra'')  in which such cohomology
theories are represented by an object $E$ in the sense that for any
based $G$-space $X$ we have 
$E_G^*(X)=[X,E]_G^*$. The category of spectra has the advantage of a much
richer structure, and in particular one can do homotopy theory in it. 
The spectra $E$ correspond to cohomology theories $E_G^*(\cdot)$ and 
spaces of natural transformations are also calculated as homotopy
classes of maps:  
$\mathrm{Nat}(E_G^*(\cdot),F_G^*(\cdot))=[E,F]^G_*. $ 
 
However, the category of free $G$-spectra remains almost as complicated as the
category of free $G$-spaces, so one does not expect to have a complete
global understanding in this generality. 
To simplify things further we assume that the cohomology theories take
values in rational vector spaces, and these are correspondingly represented by free
rational $G$-spectra. We prove in this paper that  this category is simple enough
to have a purely algebraic model. Some special cases are relatively elementary. By
Serre's early work, if $G$ is trivial,  then rational cohomology
theories correspond to graded rational vector spaces, and if $G$ is finite
then rational cohomology theories on free $G$-spaces correspond to graded $\Q G$-modules, but for
infinite compact Lie groups, the situation is more complicated. 

\subsection{Results}
We have previously given a small and concrete model of free rational 
$G$-spectra when $G$ is a connected compact Lie group \cite{gfreeq}. 
The main  result of the present paper extends this to general compact Lie
groups, but perhaps more interesting is  the new method 
(essentially that of \cite{tnq3}), which involves fewer equivalences and better 
respects multiplicative structures. Furthermore, some readers may find
it helpful to see the method of \cite{tnq3} implemented in the present 
simple context. 

The case of free $G$-spectra has the attraction that it is 
rather easy to describe both the homotopy category of free $G$-spectra
and also the algebraic model. The homotopy
category coincides with the category of rational cohomology 
theories on free $G$-spaces; better still, on free $G$-spaces an equivariant
cohomology theory is the same as one in the naive sense 
(i.e., a contravariant functor satisfying the Eilenberg-Steenrod axioms and the wedge
axiom). 

To describe the algebraic model, we suppose $G$ has identity component
$N$ and component group $W=G/N$. Note that $W$ acts on the 
polynomial ring $H^*(BN)$ by ring isomorphisms, and we write 
$\HBN$ to advertise the $W$ action. We may then form 
the twisted group ring $\HBNW$. 
A module over this ring is said to be a torsion module if it is torsion
as a module over the polynomial ring $\HBN$. The algebraic model consists 
of differential graded torsion  modules over $\HBNW$.

\begin{thm}
\label{thm:main}
For any compact Lie group  $G$, with identity component $N$ and
component group $W=G/N$, there is a Quillen equivalence
$$\freeGspectra /\Q \simeq \torsmodcat{\HBNW}  $$
of model categories. In particular their derived  categories are 
equivalent
$$\mbox{$\Q$-cohomology-theories-on-free-$G$-spaces}=Ho(\freeGspectra /\Q) \simeq D(\torsmodcat{\HBNW} ) $$
as triangulated categories. 
\end{thm}

Notice that the algebraic model does not detect the fact that the 
extension 
$$1\lra N \lra G\lra W\lra 1$$
need not be split. For example both free $O(2)$-spectra and free
$Pin (2)$-spectra are equivalent to torsion modules over 
the twisted polynomial ring $\Q [c][W]$ where $W$ is a group of 
order 2 acting to negate $c$. This should not be surprising, since
the 2:1 map $Pin(2) \lra O(2)$ induces a rational equivalence on 
categories of free spectra.

\subsection{Conventions}
Certain conventions are in force throughout the paper. The most important 
is that {\em everything is rational}: henceforth all  spectra  and homology 
theories are rationalized without
comment.  For example, the category of free rational $G$-spectra will now 
be denoted `$\freeGspectra$'.  We also use the standard conventions that
`DG' abbreviates `differential graded'.
We focus on homological (lower) degrees, with differentials reducing degrees;
for clarity, cohomological (upper) degrees are called {\em codegrees} and 
are converted to degrees  by negation in the usual way.
Finally, we write $H^*(X)$ for the unreduced cohomology of a space $X$
with rational coefficients.

\section{The proof}

\subsection{Organization of the paper}
Most of the rest of the paper is devoted to establishing the following
sequence of Quillen equivalences, several of which are themselves
zig-zags.  We will repeatedly use cellularization (right localization)
to focus attention on the free part of the model, and the
Cellularization Principle of \cite{GScellprinciple} (quoted in the
relevant special case in Appendix A).
The cellularizations are all with respect to the images of
the cell $G_+$; using the Adams spectral sequence of Section
\ref{sec:ASS}, these are recognized by their homotopy, and hence behave
as expected under the functors between the different models below. 
\begin{multline*}
\freeGspectra 
\stackrel{(1)} \simeq \cellmodcatG{DEG_+}
\stackrel{(2)} \simeq \cellmodcatW{\DBNp} \\
\stackrel{(3)} \simeq \cellmodcatQW{\CBN}
\stackrel{(4)} \simeq \cellmodcat{\HBNW}
\stackrel{(5)} \simeq \torsmodcat{\HBNW} 
\end{multline*}

To start with, we use the $G_+$-cellularization of spectra as a model
for free $G$-spectra. Now,  
Equivalence (1) is obtained from the change of rings adjunction
arising from the map $\bS \lra DEG_+$ of ring $G$-spectra 
(where $DEG_+=F(EG_+, \bS)$ is the functional dual of $EG_+$) by 
cellularizing with respect to $G_+$. This is described in Section \ref{sec:DEGmod}. 

Equivalence (2) reduces to a finite group of equivariance. It is
obtained by passage to $N$-fixed points; $\DBNp$ is the $W$-spectrum 
$(DEG_+)^N$, with the tilde included as a reminder that $W$ is
acting. This is  a form of Eilenberg-Moore equivalence and is discussed in 
Sections \ref{sec:Nfixed} and \ref{sec:EilenbergMoore} (see
\cite{GSmodulefps} for a more general discussion).  

Equivalence (3) is the big step from topology to algebra and this step is described in Section 
\ref{sec:topalg} calling on the results of \cite{s-alg}. In this case the ring
$W$-spectrum $\DBNp$ is an algebra over the rational Eilenberg-MacLane
spectrum, and hence equivalent to a $\Q W$-algebra, which we call
$\CBN$ because it is a DGA with cohomology $\HBN$. Since
we are working over the rationals, this may be taken to be commutative.
A $\CBN$-module in $\Q W$-modules is the same as
a module over the twisted group ring $\CBNW$, and we will use language
from the latter point of view. 

Equivalence (4) moves from a differential graded algebra to an
ordinary graded ring by a little formality argument described in 
Section \ref{sec:formality}. It is basically the
usual argument that commutative polynomial rings are intrinsically
formal, but a little care is needed to deal with the representations. 

Equivalence (5) is a change of model which means that cellularization
at the model category level is replaced by the use of a more
economical underlying category. This is described in Section
\ref{sec:algmodels}. 

\subsection{Relationship to other results}

We should comment on the relationship between the strategy implemented
here and that used for free spectra in  \cite{gfreeq}. Both strategies
start with a category of $G$-spectra and end with the same purely algebraic
category, and the connection in both relies on finding an intermediate
category which is visibly rigid in the sense that it is determined
by its homotopy category (the archetype of this is the category 
of modules over a commutative DGA with polynomial cohomology). 
With some additional effort, the methods of \cite{gfreeq} can be
applied to prove the main equivalence of Theorem \ref{thm:main}. 

The difference comes in the route taken. Roughly speaking,
the strategy in \cite{gfreeq} is to move to non-equivariant spectra as
soon as possible, whereas that adopted here is to keep working in 
the ambient category of $G$-spectra for as long as possible. 
The present method appears to have several advantages. It uses fewer
steps, and (although we do not pursue it here) the monoidal structures 
are visible throughout. 

We present the argument as briefly as possible, so as to highlight the line
of argument. The technical ingredients can be found in
\cite{GScellprinciple} and \cite{GSmodulefps} (a more condensed
account all in one place can be found in  \cite{tnq3}). 

\section{Modules over $DEG_+$.}
\label{sec:DEGmod}

To start we need a model for free $G$-spectra; there are several 
Quillen equivalent alternatives (see \cite[Section 3]{gfreeq} for
discussion).  For definiteness, we start with orthogonal $G$-spectra
\cite{MM}, and use the $G_+$-cellularization
of the category of $\bS$-module $G$-spectra, where $\bS$ is 
a strictly commutative model for the rationalized sphere spectrum.
Next, we explain that since $EG$ is a $G$-space, we have a strictly
cocommutative diagonal $EG\lra EG\times EG$ and hence the functional
dual
$DEG_+=F(EG_+,\bS)$ becomes a commutative ring $G$-spectrum. The
completion map 
$$\bS =F(\bS, \bS )\lra F(EG_+, \bS)=DEG_+$$
then gives a map of ring spectra. Accordingly we have a Quillen
adjunction given by extension and restriction of scalars, with counit
given by the action map
$$DEG_+\sm X \lra X$$
 for $DEG_+$-modules $X$ and unit
$$Y\lra DEG_+\sm Y. $$
Noting that the $\bS$-module $G_+$ is taken to $DEG_+ \sm G_+\simeq
G_+$, we continue to write $G_+$ for the image cell. 
We note that both the derived unit and counit are non-equivariant equivalences
and hence $G_+$-equivalences. It follows from the Cellularization Principle~\ref{prop.cell}
that we have a Quillen equivalence of cellularizations
$$\freeGspectra=\cellmodcatG{\bS}\simeq \cellmodcatG{DEG_+}. $$

\section{Passage to $N$-fixed points.}
\label{sec:Nfixed}

Now we note that since the identity component $N$ of $G$ is a normal
subgroup, Lewis-May fixed points give a functor from $G$-spectra to 
$W$-spectra. We write
$$\DBNp:=(DEG_+)^N$$
for the image of $DEG_+$. It is a ring $W$-spectrum with underlying 
non-equivariant spectrum $DBN_+$, and we include the tilde in the
notation to emphasize that it will typically have a non-trivial
$W$-action. 

The Lewis-May fixed point functor takes $DEG_+$-module $G$-spectra to 
$\DBNp$-module $W$-spectra, and we denote this functor $\Psi^N$. It
has a left adjoint given by inflation and extension of scalars, and as
usual we suppress the notation for inflation. The adjunction is
discussed at greater length in \cite{GSmodulefps} or \cite[Section
11]{tnq3}. 

Once again we have a Quillen adjunction 
$$\adjunction{DEG_+\sm_{\DBNp}}
{\modcatW{\DBNp}}
{\modcatG{DEG_+}}
{\Psi^N}$$
The Wirthm\"uller   $(G_+)^N\simeq \Sigma^dW_+$
\cite[II.6.3]{LMS(M)} gives the image of the cell $G_+$ as a $W$-spectrum, where $d$
is the dimension of $G$, and the module structure is unique by
Corollary \ref{cor:cellisunique}.
By the Cellularization Principle (Proposition \ref{prop.cell}), 
to see that we get a Quillen equivalence after cellularization,  we
need only check that the unit and counit are derived equivalences on
the cells.

For any $N$-free $G$-space $Y$ the counit is a map
$$DEG_+\sm_{\DBNp} Y^N \simeq DEG_+\sm_{\DBNp} \Sigma^d Y/N 
\lra Y,   $$
and we are interested in the special case $Y=G_+$. 
This counit is an equivalence for any $N$-free $Y$ by the Eilenberg-Moore theorem since $N$ is a
connected group, but we give the complete proof in Section \ref{sec:EilenbergMoore}, since it is
especially simple in the rational case. 
It then follows that the unit
$$W_+ \lra (DEG_+\sm_{\DBNp} W_+)^N\simeq (\Sigma^{-d}G_+)^N$$
is also an equivalence.   By the Cellularization Principle \ref{prop.cell}, we thus have a Quillen equivalence on the cellularizations
$$\cellmodcatG{DEG_+}
\simeq \cellmodcatW{\DBNp}.$$

\section{The Eilenberg-Moore argument}
\label{sec:EilenbergMoore}
We give a self contained argument for the Eilenberg-Moore equivalence in the previous section. 

\begin{prop} For any $N$-free $DEG_+$-module
$G$-spectrum $Y$ the counit  of the fixed point adjunction 
$$DEG_+\sm_{\DBNp} Y^N \simeq DEG_+\sm_{\DBNp} \Sigma^d Y/N 
\lra Y  $$
is a weak equivalence.
\end{prop}
\begin{proof}
Consider the map 
$$\epsilon: DEG_+\sm_{\DBNp} Y^N \lra Y$$
for arbitrary $DEG_+$-modules $Y$. This is a map of $G$-spectra, and it is a weak
equivalence in our model of free $G$-spectra provided
it is a non-equivariant equivalence. This means that it suffices to
argue purely $N$-equivariantly in showing that the counit is a weak
equivalence. Accordingly it is enough to argue entirely with
$N$-spectra, which we do for the remainder of the proof, so that we
have the $N$-map
$$\epsilon: DEN_+\sm_{\DBNp} Y^N \lra Y. $$

Now note that the class of $Y$ for which the counit is an equivalence is closed under
cofibre sequences, retracts and arbitrary wedges. The map $\epsilon$ is tautologously
an equivalence for the $DEN_+$-module $Y=DEN_+$ itself. It follows
that it suffices to show that $DEN_+$ builds the free cell $N_+$,
since all $N$-free spectra are built from  $N_+$. 

It remains to show that $Y=N_+$ is built from $DEN_+$. The basic idea
is to use the standard Koszul resolution, but the implementation of
the idea is complicated by some dualization. We recall that $R=H^*(BN)$ is a
polynomial ring, so that if we choose polynomial generators $x_1,
\ldots , x_r$  we may form a Koszul complex, and hence an exact
sequence
$$0\lra F_r\lra F_{r-1}\lra \cdots \lra F_1 \lra F_0 \lra \Q\lra 0
. $$
Here $F_s$ is a free module on generators corresponding to $s$-fold 
products of the generators: $F_s=\bigoplus_{|A|=s}x^A R$ (where $A$
runs over subsets of $\{ 1, 2, \ldots , r\}$, and  $x^A$
is simply designed to keep track of the degrees).  Of course, since 
$\pi_*^N(DEN_+)=R$, we can realize the modules with $DEN_+$-module
$N$-spectra  by taking
$$\F_s=\bigvee_{|A|=s} x^ADEN_+, $$
so that 
$$F_s=\pi^N_*(\F_s). $$ 
It is also easy to realize the maps in the exact sequence in
$DEN_+$-modules and one may go further to realize an entire filtered
spectrum in a standard way. 

However it is convenient to do a little more, and  find the
pre-dual. In the algebraic world, $F_s\cong (E_s)^{\vee}$ for a suitable
module $E_s$; indeed we may take $E_s=(F_s)^{\vee}$.  Taking duals
throughout, we have a resolution
$$0\lla E_r\lla E_{r-1}\lla \cdots \lla E_1 \lla E_0 \lla \Q\lla 0
 $$
by injective $H^*(BN)$-modules. 

 In the topological  world $\pi^N_*(N_+)=\Sigma^d \Q$, $\pi^N_*(EN_+)=\Sigma^d R^{\vee}$, and we
may realize $\Sigma^d E_s$ with $\E_s=\bigvee_{|A|=s}x^{-A}EN_+$. In the usual
way we form a tower
$$\diagram
Y_0\dto & Y_1\dto \lto&. \lto & \cdots & Y_{r-1}\dto \lto& 
Y_r\dto \lto &Y_{r+1} \lto \\
\E_0 &\Sigma^{-1}\E_1&&&\Sigma^{-r+1}\E_{r-1}&\Sigma^{-r}\E_r
\enddiagram$$
so that $Y_0=N_+$ and the composites $\Sigma^{-s+1}\E_{s-1}\lra \Sigma Y_s\lra
\Sigma^{-s+1}\E_s$ are the ($d$-fold suspension of the) maps in
the exact sequence. In more detail, we start with 
$Y_0=N_+$ and find a map $N_+\lra \E_0$ realizing $\Q\lra
E_0$, and take $Y_1$ to be the fibre. Now $\pi^N_*(\Sigma Y_1)=\im (E_0\lra
E_1)$
and we may realize $\pi_*^N(\Sigma Y_1)\lra E_1=\pi^N_*(\E_1)$ by a map 
$\Sigma Y_1 \lra \E_1$.  

Taking functional duals we get the required resolution by
$DEN_+$-modules.  The point of working with pre-duals is that we have a 
diagram of $N$-free spectra. 
Thus, the fact that $\pi^N_*(Y_{r+1})=0$ 
is sufficient to show that 
$Y_{r+1}\simeq *$ \cite[6.1]{gfreeq}.
Now we may form
the dual tower, 
$$\diagram
Y^0\dto & Y^1\dto \lto&. \lto & \cdots & Y^{r-1}\dto \lto& 
Y^r\dto \lto &Y^{r+1} \lto \\
\Sigma^{1-0} \E_0 &\Sigma^{1-1}\E_1&&&\Sigma^{1-r+1}\E_{r-1}&\Sigma^{1-r}\E_r
\enddiagram$$
defined by the cofibre sequences
$$Y_s\lra Y_0\lra Y^s. $$
Here we start with $Y^0\simeq *$ and build up 
$Y^{r+1}\simeq Y_0=N_+$ by 
the $r+1$ cofibre sequences 
$$\Sigma^{-s} \E_s\lra Y^{s+1}\lra Y^s. $$
Dualizing,  we have a corresponding construction
 with $DY^0\simeq *$ and build up 
  $DY^{r+1}\simeq DY_0\simeq 
   DN_{+} \simeq \Sigma^{-d} N_+$ (using the Wirthm\"uller isomorphism for the last equivalence) by 
the $r+1$ cofibre sequences  
$$\Sigma^{s} D\E_s\lla DY^{s+1}\lla DY^s. $$

This shows that $N_+$ is finitely built from $DEN_+$, which completes
the proof. 
%
\end{proof}
\section{From topology to algebra}
\label{sec:topalg}

We now have a commutative ring $W$-spectrum $\DBNp$, and by
\cite[1.2]{s-alg}, since we are working over the rationals, this
corresponds to a commutative monoid in the category of differential 
graded  $\Q W$-modules.  We write $\CBN$ for
this DGA since its cohomology is
$\HBN$.   Furthermore, we have Quillen equivalences between 
the categories of modules by \cite[2.15]{s-alg} and between their cellularizations by Proposition \ref{prop.cell} 

$$\cellmodcatW{\DBNp} \\
 \simeq 
\mbox{cell-$\CBN$-mod-$\Q[W]$-mod}.$$

Finally, for convenience we re-express this category of modules. 
Indeed, the category of $\CBN$-modules in $\Q W$-modules is the 
same as the category of $\CBNW$-modules in $\Q$-modules, where $\CBNW$
is the twisted group ring, so that we have an isomorphism. 

$$\mbox{cell-$\CBN$-mod-$ \Q[W]$-mod} \iso
\cellmodcatQW{\CBN}.$$

\section{Formality}
\label{sec:formality}

Next we replace $\CBN$ by its homology. To start with, a classical
result of Borel 
states $\HBN$ is a
polynomial algebra on even degree generators. Furthermore, if we
regard  it as a $W$-module, it is a symmetric algebra on the finite
dimensional, evenly graded $\Q W$-submodule $V=Q\HBN$. We will argue that 
there is  a copy of $V$ inside the cycles of $\CBN$. This gives a
chain map $V\lra Z\CBN \lra \CBN$ of $\Q W$-modules. 
Since $\CBN$ is commutative, the universal property of the symmetric
algebra gives a map
$$\HBN =\symm (V) \lra Z\CBN \lra \CBN$$
of differential graded $\Q W$-algebras, and it is a homology isomorphism by construction. 
This then gives a Quillen equivalence
$$\modcat{\HBNW}\simeq \modcat{\CBNW}$$
and hence also an equivalence of cellularized categories by Proposition~\ref{prop.cell}, 
where the generating cell $\Q W$ is characterized by its homology by Corollary \ref{cor:cellisunique}. 

To construct the map we work with increasing codegrees. Since $V$ is
positively cograded, we can start with the zero map in degree 0. 
When we reach codegree $n$ we have an epimorphism  
$$Z^n\CBN \lra H^n(\tilde{B}N)\lra Q^n \HBN =V^n, $$
of $\Q W$-modules. By Maschke's theorem this splits to give the 
required $\Q W$-map $V^n \lra Z^n \CBN$. Since $V$ is concentrated in
finitely many degrees, this process will be complete in finitely many
steps. 

\section{Change of algebraic models. }
\label{sec:algmodels}

The last equivalence  changes from a model with underlying category of
DG $\HBNW$-modules (and cellular equivalences as weak equivalences)
to a model with underlying category the DG torsion-$\HBNW$-modules
(and homology isomorphisms as weak equivalences). 

In fact the previous Quillen equivalence leaves us with the $\Q
W$-cellularization of the projective model structure on $\HBNW$-modules; see \cite[2.9]{s-alg} or \cite[Section 7]{johnson-morita}.  For the next step, we need the
injective model structure on $\HBNW$-modules with weak equivalences the homology isomorphisms and cofibrations the monomorphisms; see~\cite[Section 8.C]{gfreeq}) or more generally, \cite[3.6]{hess-shipley}.   Using 
the identity functors there is a Quillen equivalence between the $\Q
W$-cellularizations of the projective and injective model structures
on $\HBNW$-modules by the Cellularization Principle~\ref{prop.cell}. Now if we let $\fm$ denote the maximal ideal of $\HBN$,  the $\fm$-power torsion functor
$$\Gamma_{\fm}M:=\{ x\in M \st \fm^sx=0 \mbox{ for } s>> 0\}$$
is right adjoint to the inclusion of the torsion modules: 
$$\adjunction{i}
{\torsmodcat{\HBNW}}
{\modcat{\HBNW}_{inj}}
{\Gamma_{\fm}}. $$
We next verify that this adjunction induces a Quillen equivalence between 
the $\Q W$-cellularized
injective model category and the injective model structure on torsion modules from~\cite[8.6]{gfreeq}
with weak equivalences the homology isomorphisms and cofibrations the monomorphisms. 

First note that the inclusion of the torsion modules into the
injective model structure $\modcat{\HBNW}_{inj}$ (before cellularization)
is a left Quillen functor since it preserves homology isomorphisms and
monomorphisms.  We show in Corollary \ref{cor:cellisunique} that the generating
cell $\Q W$ is characterized by its homology, so we see that the $(i, \Gamma)$ 
adjunction preserves $\Q W$ up to equivalence.   Hence, by the
Cellularization Principle~\ref{prop.cell}, there is an induced Quillen equivalence between the associated $\Q W$ cellularized model structures.   
$$  {\mbox{cell-}{\torsmodcat{\HBNW}}} \simeq
{\cellmodcat{\HBNW}_{inj}}$$
The cellular weak equivalences detected by $\Hom (\Q W, \cdot)$ are
precisely the homology isomorphisms on torsion modules, so the
cellularized model structure on the torsion modules agrees with the
original injective model structure.    Thus, we have direct Quillen equivalences

$${\torsmodcat{\HBNW}} \simeq
{\cellmodcat{\HBNW}_{inj}} \simeq {\cellmodcat{\HBNW}_{proj}}.$$

\section{The Adams spectral sequence.}
\label{sec:ASS}
The homotopy groups $\pi^N_*$  may be used as the basis of an 
Adams spectral sequence for calculating maps between free rational
$G$-spectra. If $X$ is a $G$-space then $W$ acts on $\pi^N_*(X)$, and 
if $X$ is $N$-free  then $\HBN$ acts on $\pi^N_*(X)=H_*(X/N)$ by cap
product; 
these structures interact to give the structure of an $\HBNW$-module. 
Finally, since homotopy elements are supported on finite subspectra 
(whose homotopy is bounded below),  the module is a torsion module. 

\begin{thm} 
For any free $G$-spectra $X$ and $Y$ there is a natural
Adams spectral sequence
$$\Ext_{\HBNW}^{*,*}(\pi^N_*(X) , \pi^N_*(Y))\Rightarrow [X,Y]^{G}_*.$$
It is a finite spectral sequence concentrated in rows $0$ to $r$ 
and strongly convergent for all $X$ and $Y$. 
\end{thm}

\begin{proof}
The proof is standard. First we observe that enough torsion modules
are realizable, since $\pi^N_*(EG_+\sm G/N_+)\cong \Sigma^d
\HBN [W]^{\vee}$. Next we observe that for any free $G$-spectrum $X$
$$\pi^N_*:[X, EG_+\sm G/N_+]^G_*\lra
\Hom_{\HBN [W]}(\pi^N_*(X), \pi^N_*(EG_+\sm G/N_+))$$
is an isomorphism. Indeed, we may use a change of groups isomorphism
on the domain and codomain and reduce to showing  that 
$$\pi^N_*:[X, EN_+]^N_*\lra
\Hom_{\HBN}(\pi^N_*(X), \pi^N_*(EN_+))$$
is an isomorphism. This is a special case of  \cite[6.1]{gfreeq}.

This is enough to construct the Adams spectral sequence, and identify
the $E_2$-term. For convergence, we need only show that if $X$ is
$G$-free and $\pi_*^N(X)=0$ then $X\simeq *$. By \cite[6.1]{gfreeq} we
know that $\pi^N_*(X)=0$ implies that $X$ is $N$-contractible, or equivalently that $X\sm
G/N_+\simeq *$. It follows that $X\sm EG/N_+\simeq *$ and hence
we have equivalences 
$$X=X\sm S^0 \simeq  X\sm \widetilde{E}G/N\simeq *.$$
The first equivalence is because $X\sm EG/N_+\simeq *$ and the 
second is because $X$ is $G$-free and 
$\widetilde{E}G/N$ is non-equivariantly contractible. 
\end{proof}

Apart from giving a calculational tool, this result makes plausible the main theorem of the
present paper. Nonetheless, it appears that the only way we explicitly use the
Adams spectral sequence  is in the fact that cells are characterized by their
homology. 

\begin{cor} 
\label{cor:cellisunique}
If $X$ is a free $G$-spectrum with $\pi^N_*(X)\cong
\pi^N_*(G_+)=\Sigma^d \Q W$ then 
$X\simeq G_+$.
\end{cor}

\begin{proof}
The $E_2$-term of the Adams spectral sequence for calculating maps
between $G_+$ and $X$  is 
$$\Ext^{*,*}_{\HBNW}(\Q W, \Q W)=(\Lambda V)[W]$$ with
$V=Q\HBN$ the indecomposables. A degree $-i$ submodule of $V$ gives rise to a bicodegree
$(1,-i)$ submodule of the Ext group, and so by degree, the bottom copy
of $\Q [W]$ consists of infinite cycles. It follows that the identity map in
$\pi^N_*$ lifts to a map between spectra. 
This gives maps $G_+\lra X$ and $X\lra G_+$ whose composites in either
order are isomorphisms in $\pi^N_*$. By the convergence
of  the Adams spectral sequence this is an equivalence. 
\end{proof}

In the present paper, we often need to know how our chosen cells
behave under functors between model categories. We  apply the 
corollary repeatedly to see that each cell maps to the obvious object up to 
equivalence. 

\begin{appendix}
\section{Cellularizations}

The basic reference for cellularization (or right localization) in model structures is~\cite{hh}.   Our convention throughout this paper is to refer to the cellularization of a stable model category $\M$ with respect to the set of all suspensions of an object $\{ \Sigma^{i} A\}_{i\in \Z}$ as {\em the cellularization of $\M$ with respect to $A$} and write $A$-cell-$\M$.
In this case $A$-cell-$\M$ is again a stable model category;
see~\cite[4.6]{barnes-roitzheim-localizations}.

We now recall the Cellularization Principle from
\cite{GScellprinciple}  (or \cite[Appendix A]{tnq3})
 which we use to produce Quillen equivalences of cellularized model categories.   

\begin{prop}\label{prop.cell} \cite{GScellprinciple} {\bf (The Cellularization Principle.)}
Let $\M$ and $\N$ be right proper, stable, cellular model
categories with $L: \M \to \N$ a Quillen 
adjunction with right adjoint $R$.  
Let $\underbar L$ and $\underbar R$ denote the associated derived functors. 
\begin{enumerate}
\item Let $A$ be a homotopically small object in $\M$ such that
  $\underbar L A$ is homotopically small and $A \to \underbar R \underbar L A$ is a weak equivalence.
Then $L$ and $R$
induce a Quillen equivalence between the $A$-cellularization
of $\M$ and the $\underbar L A$-cellularization of $\N$.
\[\mbox{$A$-cell-$\M$}\simeq_{Q} \mbox{$\underbar L A$-cell-$\N$}\]
\item Let $B$ be a homotopically small object in $\N$ such that
  $\underbar R B $ is homotopically small and
$\underbar L \underbar R B \to B$ is a weak equivalence.
Then $L$ and $R$
induce a Quillen equivalence between the $B$-cellularization
of $\N$ and the $\underbar R B$-cellularization of $\M$.
\[ \mbox{$\underbar R B$-cell-$\M$} \simeq_{Q} \mbox{$B$-cell-$\N$}\]
\end{enumerate} 
\end{prop}

\end{appendix}

\end{document}